\documentclass[12pt]{amsart}
\usepackage{amssymb}
\usepackage{url}
\usepackage{graphicx}
\usepackage{comment}
\usepackage{color}
\theoremstyle{plain}

\newtheorem{theorem}{Theorem}[section]

\newtheorem{cor}[theorem]{Corollary}

\newtheorem{thm}[theorem]{Theorem}
\newtheorem{defn}{Definition}
\newtheorem{remark}{Remark}[section]

\def\RR{\mathbb R}
\def\QQ{\mathbb Q}

\title[Enumerating]{Enumerating Palindromes and Primitives in Rank Two
 Free Groups  }

\author{Jane Gilman and Linda Keen }

\address{Department of  Mathematics, Rutgers University, Newark, NJ
 07079} \email{gilman@rutgers.edu; jgilman@nsf.gov} \thanks{Research supported in part by Rutgers Research Council, Yale University, \& NSF}
\address{Department of  Mathematics,  CUNY, Lehman College and Graduate
Center,  Bronx, 10468} \email{linda.keen@lehman.cuny.edu}
\thanks{Research supported in part by PSC-CUNY research grant}

\begin{document}

\begin{abstract}

Let $F$  be a free group of rank two. An element of $F$ is {\sl primitive} if it, along with another group element, generates the group. If $F= \langle A,B \rangle$, then a word $W(A,B)$, in $A$ and $B$, is a {\sl palindrome} if it reads the same forwards and backwards.
It is known that in a rank two free group, for any fixed set of two generators a primitive element will be conjugate either to a palindrome or to the product of two palindromes, but known iteration schemes for all primitive words give only a representative for the conjugacy class. Here we derive a new iteration scheme that gives either the unique palindrome in the conjugacy class or expresses the word as  a unique
product of two unique palindromes that have already appeared in the scheme. We denote these words by $E_{p/q}$ where $p/q$
is rational number expressed in lowest terms. We prove that $E_{p/q}$
is a palindrome if $pq$ is even and the unique product of two unique
palindromes if $pq$ is odd. We prove that  the pair $(X,Y)$ (or $(X^{-1},Y^{-1})$) generates the group if and only if $X$  is conjugate to $E_{p/q}$ and $Y$   is conjugate to $E_{r/s}$ where
 $|ps-rq|=1$. This
improves a previously known result that held only for $pq$ and
$rs$ both even. The derivation of the enumeration scheme also gives
a new proof of the known results about primitive words.
\end{abstract}

\maketitle

\section{Introduction}

Let $F= \langle A,B \rangle$ be a rank two free group.
It is well known that up to conjugacy and inverses, primitive elements  in a
rank two free group can be indexed by the rational numbers
  and that pairs of primitive words that
generate the group can be obtained by the number theory of the Farey
tessellation of the hyperbolic plane.
It is also well known that up to conjugacy, a primitive word can
always be written as either a palindrome or a  product of two
 palindromes and that certain pairs of palindromes will generate the
group. \cite{BShpilT, Piggott}

 In this paper we give  new proofs of the above results. The proofs yield a
 new enumerative scheme for conjugacy
classes of primitive words, still indexed by the rationals (Theorem
\ref{thmmain}). We denote
the words representing each conjugacy class by $E_{p/q}$. In
addition to proving that the enumeration scheme gives a unique
representative for each conjugacy class containing a primitive, we
prove that the words in this scheme are either palindromes or the
canonically defined product  of a pair of palindromes that have
already appeared in the scheme and thus
give a new proof of the palindrome/product result.  Further we show
 that if  $X$ and $Y$ are primitive, then up to replacing one and/or both by their inverse, there are  words $E_{p/q}$ and $ E_{r/s}$ respectively conjugate to $X$ and $Y$, and these generate the group
 if and only if
$|pq-rs|= 1$ (Theorem~\ref{mainpos}).   This improves the previous known
result that held only for pairs of palindromes.

 In this paper,  we use
  continued fractions  and  the Farey tessellation of the hyperbolic plane
 to
find the enumeration scheme and to prove that it actually enumerates
all primitives and all primitive pairs.

Pairs of primitives that generate the group  and  their
relation to continued fractions and  the Farey tessellation of the hyperbolic plane
arise in the discreteness algorithm
for $PSL(2,\mathbb{R})$ representations of two generator groups
\cite{G,G1,GM, YCJ}. This new enumerative scheme is useful in extending
discreteness criteria to $PSL(2, \mathbb{C})$ representations where
the hyperbolic geometry of palindromes plays an important role
\cite{GKgeom, GKMORE}.

\section{The Main Result}

We are able to state and use our main result with very little
notation, only the definition of continued fractions. Namely, we let
$p$ and $q$ be relative prime integers positive integers.  Write
$$ {\frac{p}{q}}
= a_0 +{1 \over {a_1 + {1 \over {a_2 + {1 \over a_3 + \dots +{1
\over a_k }}}}}}
 =[a_0, a_1, \ldots,a_k]$$
 where the $a_i$ are integers with $a_j > 0$, $j=1\ldots k$, $a_0 \geq 0$.

\vskip .2in
 \centerline{\bf Enumeration Scheme for positive rationals}
 \vskip .1in
 Set
$$  E_{0/1}=A^{-1},  E_{1/0}= B, \mbox{ and } E_{1/1}=BA^{-1}.$$

Suppose $p/q$ has continued fraction expansion $[a_0, a_1, ..., a_{k-1},
a_k]$.  Consider the two rationals defined by the continued fractions $ [a_0, a_1,...,a_{k-1}]$ and $[a_0,a_1,...,a_{k-1}, a_k-1]$.  One is smaller than
$p/q$ and the other is larger;  call the smaller one $m/n$ and the larger one $r/s$ so that $m/n < p/q <r/s$.  The induction step in the scheme is given by

\vskip .1in
 \noindent{\bf Case 1} $pq$ - odd: $$E_{p/q}=E_{r/s}E_{m/n}.$$

\noindent{\bf Case 2} $pq$ - even: $$E_{p/q} = E_{m/n}E_{r/s}.$$

\vskip .1in

We have a similar scheme for negative rationals described in
section~\ref{section:enumerating}.  With both schemes we can state
our main result as

\begin{thm}\label{thmmain} {\bf \rm {\bf (Enumeration of Primitives by Rationals)}}\\
            Up to conjugacy and taking formal inverses, the primitive elements of a two generator free group
can be enumerated by the rationals  using continued fraction
expansions. The resulting primitive words are  cyclically reduced and denoted by $E_{p/q}$ so that:

\begin{itemize}
\item  For $pq$ even,  $E_{p/q}$ is a palindrome. It  is the unique palindrome in its
conjugacy class.

\item For $pq$ odd,  $E_{p/q}$ is a product of palindromes that
  have already appeared in the scheme; that is,  $E_{p/q} = E_{m/n}E_{r/s}$ and both $E_{m/n}$ and
$E_{r/s}$ are palindromes.
\end{itemize}

\end{thm}

\begin{remark} Note that although there are several ways a word in the $pq$ odd conjugacy
class can be factored as products of palindromes, in this theorem we
specifically choose the unique factorization for $E_{p/q}$  that
makes the enumeration scheme work.
\end{remark}

In addition we have
\begin{thm} \label{mainpos} Let $\{E_{p/q}\}$  denote the words in the
enumeration scheme for rationals.  If $(p/q,p'/q')$
satisfies $|pq'-qp'| =1$,  the pair $(E_{p/q}, E_{p'/q'})$ generates
the group.   Conversely,  the pair of elements $(X,Y)$ generates the group if and only if   there is a rational $p/q$ such that $X$ is conjugate to $E_{p/q}$ and a rational $p'/q'$ such that  $Y$ is  conjugate to $E_{p'/q'}$,  where  $|pq'-qp'| =1$.
\end{thm}

 These theorems will be proved in
section \ref{section:enumerating}. In order to prove the theorems and
the related results we need to review some terminology and
background.

\section{Preliminaries}

 The main object here is a two generator free
group which we denote by $F = \langle A,B \rangle$. A word $W=W(A,B)
\in F$ is, of course, an expression of the form
\begin{equation} \label{eq:form} A^{m_1}B^{n_1}A^{m_2} \cdots B^{n_r} \end{equation}
 for some set of $2r$  integers
$m_1,...,m_r,n_1,...,n_r$ with  $m_2,...,m_r, n_1,...,n_{r-1}$
  non-zero.

\begin{defn} An element of $F$ is {\em primitive} if there is
another element $V$ such that $W$ and $V$ generate $F$.
$V$ is called a {\em primitive associate} of $W$ and the unordered
pair $W$ and $V$  is called a {\em pair of primitive associates} or
a {\sl primitive pair} for short.

\end{defn}

If $F=\langle A,B \rangle$, a primitive element $W=W(A,B)$ is also referred to as a {\sl primitive word}
when we are thinking of it as a word in $A$ and $B$.

There are various versions of the following theorem in the literature, for example in \cite{BuserS, CMZ, GKHarvey, KR, KS,MKS, Malik2,OZ}.  We use a version convenient for what we need below.

\begin{thm}\label{thm:genform} A word  $W(A,B) \in F$ is primitive
if and only  if after freely reducing the word  and  interchanging  $A$ and $B$ and $p$ and $q$, it  is conjugate to a word of  the form

\begin{equation} \label{eq:form1} B^{\epsilon}A^{n_1}B^{\epsilon} \cdots A^{n_p} \end{equation}
where $\epsilon = \pm
1$, \; all of the $n_j$ have the same sign, \;  and  $|n_j|  \ge 1,  \; 1 \leq j < p, \; |n_j-n_{j+1}| \le 1$, and $j$ is taken modulo $p$.  The exponents satisfy $\Sigma_{i=1}^p n_i = q$  with  $p$ and $q$  relatively prime and $0 \leq |p/q |<1$.

Moreover, to every rational $p/q$ there is a unique primitive word of this form.
\end{thm}

The elements of the ordered set $\{n_1,...,n_p\}$  are called the {\sl primitive exponents} of the primitive word.  Note that the primitive exponents are defined only up to cyclic permutation.

In the next subsections we summarize terminology and facts about the Farey tessellation and continued fraction expansions for rational numbers. Details and proofs can be found in \cite{CS3, CS2, Series}.
See also \cite{Vin}.

\subsection{Preliminaries: The Farey Tessellation}
 In what follows when we
use $r/s$ to denote a rational number, we assume that $r$ and $s$
are integers with $s > 0$,  and that $r$ and $s$
are relatively prime, that is, that $(r,s)=1$. We let $\mathbb{Q}$
denote the rational numbers, and  we identify the rationals with
points on the extended real axis on the Riemann sphere. We use the
notation $1/0$ to denote the point at infinity.

We  need the concept of Farey addition for
fractions.

\begin{defn} If $\frac{p}{q}, \frac{r}{s} \in \QQ$ with $|ps-qr| =1$, the {\em Farey sum} is
$$\frac{p}{q} \oplus \frac{r}{s} = \frac{p+r}{q+s}$$
Two fractions are called {\em Farey neighbors} if  $|ps-qr|=1$.
\end{defn}

When we write $\frac{p}{q} \oplus \frac{r}{s} = \frac{p+r}{q+s}$ we tacitly assume
the fractions are Farey neighbors.

\begin{remark} \label{remark:nbs} It is a simple computation to see
that both pairs of fractions $$(\frac{p}{q},\frac{p}{q} \oplus
\frac{r}{s}) \mbox{  and  } (\frac{p}{q} \oplus \frac{r}{s},
\frac{r}{s})$$ are Farey neighbors if $(p/q, r/s)$ are,
and  that if $\frac{p}{q}< \frac{r}{s}$,  then $$\frac{p}{q} < \frac{p}{q} \oplus \frac{r}{s} < \frac{r}{s}.$$
 \end{remark}

   It is
easy to calculate that the Euclidean distance between finite Farey
neighbors is strictly less than one unless they are adjacent
integers. This implies that unless one of the fractions is $0/1$,
both neighbors have the same sign.

One creates the Farey diagram in the upper half-plane by marking
each fraction by a point on the real line and joining each pair of
Farey neighbors by a semi-circle orthogonal to the real line. The
point here is that because of the above properties none of the
semi-circles intersect in the upper half plane. This gives a
tessellation of the hyperbolic plane where the semi-circles joining
a pair of neighbors, together with the semi-circles joining each
member of that pair to the Farey sum  of the pair, form an ideal
hyperbolic triangle. The tessellation is called the Farey
tessellation and the vertices are precisely the points that
correspond to rational numbers.  A vertex corresponding to $p/q$ is labeled $v_{p/q}$. It is termed {\sl even} or {\sl odd}  according to the parity of $pq$. See Figure~\ref{fig:farey}

\begin{figure}[http]
\begin{center}
\includegraphics[width=5in,height=7in]
{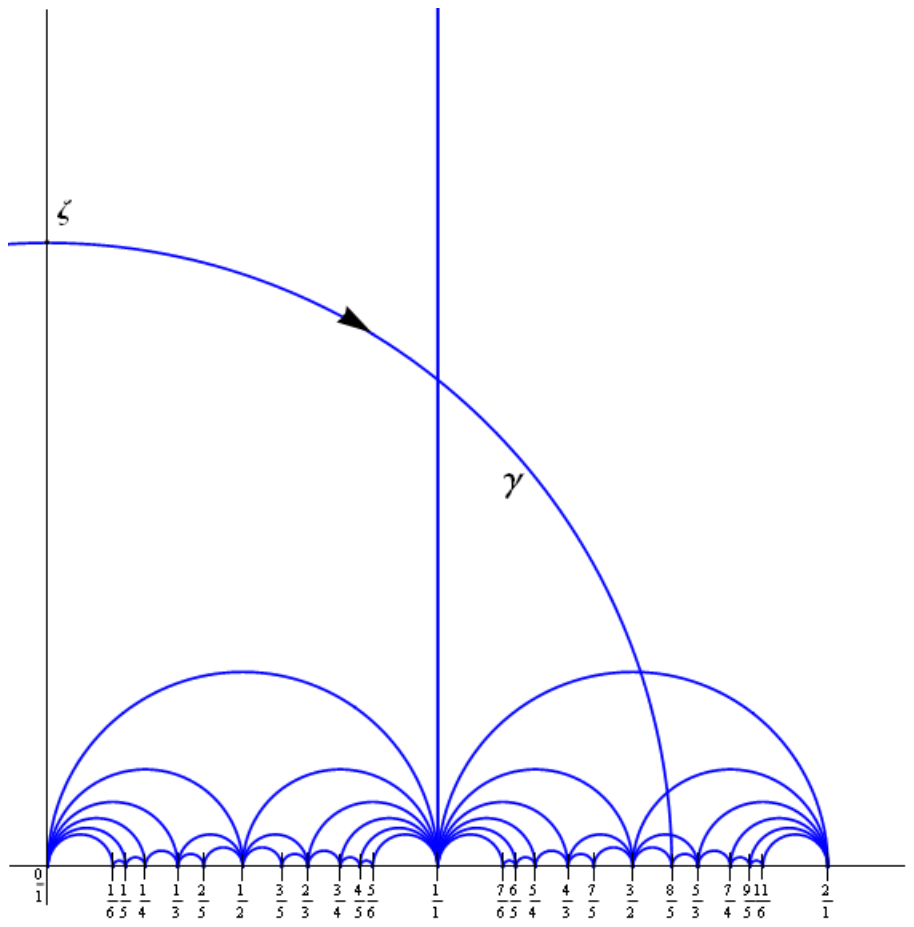}
\caption{The Farey Tesselation with the curve $\gamma$}
\label{fig:farey}
\end{center}
\end{figure}

The Farey tessellation is invariant under the semi-group generated
by $z \mapsto z+1$ and $z \mapsto 1/z$.

Fix  any point $\zeta$ on the positive imaginary axis. Given a
fraction, ${\frac{p}{q}}$, there is an oriented  hyperbolic geodesic $\gamma$
connecting  $\zeta$ to ${\frac{p}{q}}$. We assume $\gamma$ is oriented so that moving from $\zeta$ to a positive rational is the positive direction. This geodesic $\gamma$ will intersect some number of
triangles.

\begin{defn} The {\em Farey level} or the {\em level} of $p/q$,  denoted by $Lev(p/q)$, is the number of
triangles
traversed by $\gamma$ \end{defn}

Note that the definition is independent of the choice of $\zeta$ and that our definition implies $Lev(p/q)=Lev(-p/q)$.
\vskip .05in

The geodesic $\gamma$ will enter any given triangle along an edge.
This edge will connect two vertices and $\gamma$ will exit the
triangle along an edge connecting one of these two vertices and the
third vertex of the triangle, the {\sl new} vertex. Since $\gamma$
is oriented, the edge through which $\gamma$ exits a triangle is
either a left edge or a right edge depending upon whether the  edge
$\gamma$ cuts is to the right or left of the new vertex.

\begin{defn} We determine a {\em Farey sequence for ${\frac{p}{q}}$} inductively by choosing
the new vertex of the next triangle in the sequence of
triangles traversed by $\gamma$.   The sequence ends at $p/q$.  \end{defn}

Given $p/q$, we can find the smallest rational $m/n$ and the largest rational
 $r/s$ that are neighbors of $p/q$.  These neighbors  have the property
that they are the only neighbors with lower Farey level. That is,
$m/n < p/q < r/s$ and  $Lev(p/q) > Lev(m/n)$, $Lev(p/q) >Lev(r/s)$,
and if $u/v$ is any other neighbor $Lev(p/q)< Lev(u/v)$.

\begin{defn}  We call the smallest and the largest neighbors of the rational $p/q$
 the {\em distinguished neighbors} or the {\sl parents} of $p/q$.
\end{defn}

Note that we can tell whether a distinguished neighbor $r/s$ is
smaller or larger than $p/q$ by the sign of   $rq-ps$.

We emphasize that we have two different and independent orderings of
the rational numbers: the ordering as rational numbers and the
partial ordering by level. Our proofs will often use induction on
the level of the rational numbers involved as well as the order relations as rational numbers
among parents and grandparents.

\begin{remark} It follows from remark \ref{remark:nbs} that
if $m/n$ and $r/s$ are the parents  of $p/q$, then any other
neighbor of $p/q$ is of the form ${\frac{m + pt}{n + qt}}$ or
${\frac{r + pt}{s + qt}}$ for some positive integer $t$. The
neighbors of $\infty$ are precisely the set of integers.
\end{remark}

Finally we note that we can describe the Farey sequence of $p/q$ by listing
the number of successive left or right edges of the triangles that $\gamma$ crosses where a left edge means
that there is one vertex of the triangle on the left of $\gamma$ and two on the right, and a right edge means there is only one vertex of $\gamma$ on the right.
This will be a sequence of integers, the {\sl left-right sequence}
$\pm(n_0, n_1,...,n_t)$ where the integers $n_i, i > 0$ are all positive and
$n_0$ is positive or zero.  The sign in front of the sequence is positive or negative depending on the orientation of $\gamma$.

\subsection{Preliminaries: Continued Fractions} \label{sec:cf}

Farey sequences are related to continued fraction expansions of
positive fractions; they can also be related to expansions of
negative fractions.  We review the connection in part to fix our
notation precisely.   We do not use the classical notation of
\cite{HardyWright} for negative fractions, but instead use the
notation of \cite{GKnew,MKS,CS3,CS2,Wr} which is standardly used by
mathematicians working in Kleinian groups and three manifolds. This
notation reflects the symmetry about the imaginary axis in the Farey
tessellation which plays a role in our applications.   This symmetry
is built in to the semi group action on the tessellation.

For  $p/q \geq 0$
write
$$ {\frac{p}{q}}
= a_0 +{1 \over {a_1 + {1 \over {a_2 + {1 \over a_3 + \dots +{1
\over a_k }}}}}}
 =[a_0, a_1, \ldots,a_k]$$
 where $a_j > 0$, $j=1\ldots k$, $a_0 \geq 0$.
  For $0\le n \le k$ set $$\frac{p_n}{q_n}=[a_0, a_1,  \ldots,  a_n].$$

  \begin{remark} The continued fraction of a
  rational is ambiguous; the continued fractions $[a_0,a_1,...,a_n]$ and $[a_0,a_1,...,a_n-1,1]$
  both represent the same rational.   Therefore, if we have $[a_0,a_1,...,a_{n-1},1]$  we may replace
  it with $[a_0,a_1,...,a_{n-1}+1]$.  \end{remark}

\begin{remark} Note that if $\frac{p}{q}  \geq 1$ has continued fraction expansion $[a_0,a_1,...,a_n]$,  then
${\frac{q}{p}}$ has expansion $[0,a_0,...,a_n]$ while   if $\frac{p}{q}  < 1$ has continued fraction expansion $[0,a_1,...,a_n]$,  then
${\frac{q}{p}}$ has expansion $[a_1,a_2,...,a_n]$.
\end{remark}

\begin{remark} \label{remark:CFexp} The distinguished neighbors or parents of $p/q$ have continued fractions $$[a_0, a_1 \ldots, a_{k-1}] \mbox{  and   } [a_0, a_1, \ldots,
a_{k-1},a_k-1].$$
\end{remark}

 The Farey sequence contains
the approximants ${\frac{p_n}{q_n}}$  as a subsequence.   They
can be computed recursively from the continued fraction for $p/q$ as
follows:
$$p_0 = a_0,  q_0 = 1  \mbox{ and } p_1=a_0a_1+1, q_1 = a_1 $$
$$ p_j = a_j p_{j-1} + p_{j-2} \, \, , q_j = a_j q_{j-1} + q_{j-2} \, \, j=2, \dots, k.$$
One can calculate   from these recursion formulas that   the approximants are alternately to the right and left of $p/q$. The points of the
Farey sequence between ${\frac{p_j}{q_j}}$  and
${\frac{p_{j+1}}{q_{j+1}}}$ have continued fraction expansions
$$[a_0, a_1, \ldots a_{j} + 1], [a_0,a_1, \ldots, a_{j} + 2], \ldots,  [a_0,a_1, \ldots a_{j} +
a_{j+1}-1].$$ Thus $a_j$ is the number of points on the same side of $p/q$ before jumping to the other side.

We  extend the continued fraction notion to negative fractions by defining the continued fraction of $\frac{p}{q}<0$
to be the negative of that for $|\frac{p}{q}|$.  That is,  by setting

$$\frac{p}{q}=-[a_0,a_1 \ldots, a_k] = [-a_0,-a_1 \ldots, -a_k]  \mbox{ where  }
|\frac{p}{q}| = [a_0,a_1,\dots a_k].$$
We also set $Lev(p/q)=Lev(|p/q|)$.

In \cite{HardyWright} the continued fraction
$[a_0,a_1, \ldots, a_k]$  of $p/q<0$ is defined so that $a_0< 0$ is
the largest integer in $p/q$ and $[a_1, \ldots a_k]= p/q-a_0$  is the continued
fraction of a positive rational. With this notation the symmetry about
the origin which plays a role in our applications is lost.

We note that for any pair of neighbors, unless one of them is $0/1$ or $1/0$, they both have
the same sign and thus have equal $a_0$ entries.  Since we almost always work with neighbors the
difference between our notation and the classical one does not play a role.

 It is not difficult to check that the continued fraction expansion and the left-right sequence of a rational agree.

\section{Enumerating primitives: palindromes and products}
\label{section:enumerating}

In the enumeration scheme we define here, the group element $B^{-1}$ never appears.  It is therefore clear that either a word or its inverse is in the scheme, but both cannot be.

We first work with positive rationals. We do this merely for ease of
exposition and to simplify the notation.
  We  then indicate the minor changes needed for negative rationals.

 \bigskip

 \centerline{\bf Enumeration Scheme for positive rationals}
 Set
$$  E_{0/1}=A^{-1},  E_{1/0}= B, \mbox{ and } E_{1/1}=BA^{-1}.$$

If $p/q$ has continued fraction expansion $[a_0, a_1, ..., a_{k-1},a_k]$,  consider the parent fractions
  $ [a_0,_1,...,a_{k-1}]$ and $
[a_0,...,a_{k-1}, a_k-1]$ (see Remark~\ref{remark:CFexp}).   Choose labels $m/n$ and $r/s$ for the parents so that $m/n < p/q < r/s$.  Set

\vskip .2in

 \noindent{\bf Case 1} $pq$ - odd: $$E_{p/q}=E_{r/s}E_{m/n}$$

\noindent{\bf Case 2} $pq$ - even: $$E_{p/q} = E_{m/n}E_{r/s}.$$

\vskip .2in
 Note that in Case 1 ($pq$ odd)  the word indexed by the larger
fraction is on the left and in Case 2 ($pq$ even) it is on the right.

\bigskip
 \centerline{\bf Enumeration Scheme for negative rationals}
 \bigskip

 Now assume $p/q<0$.   We use the reflection in the imaginary axis to obtain the enumeration scheme.   The reflection sends $A^{-1}$ to $A$.  This  reverses the order of the distinguished neighbors on the real line.

  Set
$$  E_{0/1}=A \mbox{    and    } E_{1/0}=  B $$  These are
trivially palindromes.  At the next level we have $E_{-1/1}=BA. $

To give the induction scheme:  we assume  $m/n,r/s$ are the distinguished
neighbors of $p/q$
and  they satisfy  $m/n >
p/q > r/s$, $p/q=(m+r)/(n+s)$ and set\\

 \noindent{\bf Case 1} $pq$ - odd:
 $$E_{p/q}=E_{r/s}E_{m/n}.$$

\noindent{\bf Case 2} $pq$ - even:
$$E_{p/q} = E_{m/n}E_{r/s}.$$

Note that now in Case 1 ($pq$ odd) the word indexed by the larger fraction is on
the right and in Case 2 ($pq$ even) it is on the left.

\bigskip

\centerline{\bf Proofs of Theorem~\ref{thmmain} and~\ref{mainpos}}
\vskip .2in
Before we give the proofs we note that there are  other enumeration schemes for relating rational numbers to  primitive words  given, for example, in
\cite{CMZ, GKwords,KS}.    Some of these also use Farey neighbors to find primitive pairs.  In general, these schemes produce a different word in the conjugacy class of, or the inverse conjugacy class of, our $E_{p/q}$.

        Theorem ~\ref{mainpos} tells us that words  in the enumeration scheme labeled with neighboring Farey fractions
  give rise to primitive pairs. Note that although cyclic
permutations are obtained by conjugation, we cannot necessarily
   {\sl simultaneously} conjugate a primitive pair coming
from the other enumeration schemes to get to the corresponding primitive pair
coming from Theorem~\ref{mainpos}. Here by simultaneously we mean conjugate the two elements of the primitive pair by the same single element of the group.

\begin{proof} ({\bf proof of Theorem \ref{thmmain}})
The proof uses the connection between  continued fractions and the
Farey tessellation.  We call a rational $p/q$ or the word $E_{p/q}$ {\em odd} if $pq$ is odd and  {\em even} otherwise.
We observe that in every Farey triangle with
vertices $m/n,p/q, r/s$ one of the vertices is {\sl odd} and the
other two are {\sl even}. To see this simply use the fact that
$mq-np, ps-rq, ms-nr$ are all congruent to $1$ modulo $2$. (This
also gives the equivalence of parity cases for $pq$ and the $p+q$
used by other authors.) In a triangle where $pq$ is even, it may be that
the smaller distinguished neighbor is even and the larger odd or
vice-versa and we take this into account in discussing the
enumeration scheme. We note that in general if $X$ and $Y$  are
palindromes, then so is $(XY)^tX$ for any positive integer $t$.

We give the proof assuming $p/q>0$.  The
proof proceeds by induction on the Farey level.
The idea behind the proof is that each rational has a pair of
parents (distinguished neighbors) and each parent in turn has two parents so there are at most
four {\sl  grandparents} to consider. The parents and grandparents
may not all be distinct. The cases considered below correspond to
the possible ordering of the grandparents as rational numbers and
also the possible orders of their levels.

 To deal with negative rationals we use the reflection in the imaginary axis.     The reflection sends $A^{-1}$ to $A$.  We again have distinguished neighbors $m/n$ and $r/s$, and using the reflection our assumption is
  $m/n>p/q>r/s$.  In the statement of the theorem,  $m/n$ is now the larger neighbor.   Using our definition of the  Farey level of $p/q$ as the Farey level of $|p/q|$,  the proof is exactly the same as for positive rationals.

 \medskip

If $p/q$ is odd,  by induction we  get the product of distinguished neighbor
palindromes.

\medskip

If $p/q$ is even we  need to show that we get palindromes.
\vskip .1in

   The set up shows that we have palindromes for level $0$, ($\{0/1,1/0\}$) and the correct product for
   level $1$, $\{1/1\}$.

Assume the scheme works for all rationals with  level less than $N$
and assume $Lev(p/q)=N$. Since $m/n,r/s$   are distinguished
neighbors of $p/q$   both their levels are less than $N$.

 Suppose $p/q=[a_0,a_1, \ldots, a_k]$.  Then $Lev(p/q)=\sum_0^k a_j=N$,  and by Remark \ref{remark:CFexp} the continued fractions of the parents of $p/q$ are $$[a_0,a_1, \ldots, a_{k-1}]  \mbox{  and   } [a_0,a_1,\ldots, a_{k-1}, a_k-1].$$
Assume we are in the case where $pq$ is even.

Suppose first that  $$Lev(r/s) > Lev(m/n)$$  so that  we have

   \begin{equation} \label{eqn:levrsbigger}
   m/n=[a_0,a_1, \ldots, a_{k-1}]  \mbox{  and   } r/s=[a_0,a_1,
\ldots, a_{k-1}, a_k-1].
\end{equation}
Then the smaller
distinguished neighbor of $r/s$ is  $ m/n$ and the larger  distinguished
neighbor is
 \begin{equation}\label{eqn:firstgp} w/z=[a_0,a_1,\ldots, a_{k}-2].
\end{equation}
The smaller distinguished neighbor
of $m/n$ is
\begin{equation}\label{eqn:secondgp}u/v=[a_0, \ldots,
a_{k-2},a_{k-1}-1]\end{equation}
and the larger distinguished
neighbor is
 \begin{equation}\label{eqn:sixthgp} x/y = [ a_0, \ldots,
a_{k-2}].\end{equation}

If $rs$ is odd we have, by the induction hypothesis
$$E_{r/s}= E_{w/z}E_{m/n}$$ and
$$E_{p/q}= E_{m/n}E_{r/s}=E_{m/n}( E_{w/z}E_{m/n})$$ which
is a palindrome.

If $mn$ is odd we have, by the induction hypothesis,
$$E_{m/n} =E_{x/y}E_{u/v}$$ and by
equations~(\ref{eqn:levrsbigger}),~(\ref{eqn:firstgp}),~(\ref{eqn:secondgp})  and
(\ref{eqn:sixthgp})
$$E_{r/s}=E_{m/n}^{(a_{k}-1)} E_{x/y}=(E_{x/y}E_{u/v})^{(a_{k}-1)}E_{x/y}$$ so that
$$E_{p/q}=E_{m/n}E_{r/s}= (E_{x/y}E_{u/v})^{a_{k} }E_{x/y}$$ is a
palindrome.

\vskip .2in
    If
  $$Lev(r/s)<Lev(m/n),$$ we have
  \begin{equation}\label{eqn:levmnbigger} m/n=[a_0,a_1, \ldots, a_{k-1},a_k-1]  \mbox{  and   } r/s=[a_0,a_1,
\ldots, a_{k-1}].\end{equation} Then the larger
distinguished neighbor of $m/n$  is $ r/s$  and the smaller distinguished
neighbor is
 \begin{equation}\label{eqn:thirdgp} w/z=[a_0,a_1,
\ldots, a_{k}-2].\end{equation}
The larger distinguished neighbor
of $r/s$ is
\begin{equation}\label{eqn:fourthgp} x/y=[a_0, \ldots,
a_{k-2},a_{k-1}-1]\end{equation}
and the smaller distinguished
neighbor is \begin{equation} \label{eqn:fifthgp}u/v = [ a_0, \ldots,
a_{k-2}].\end{equation}

If $mn$ is odd we have, by the induction hypothesis
$$E_{m/n}=E_{r/s}E_{w/z}$$ and
$$E_{p/q} = E_{m/n}E_{r/s}=E_{r/s}( E_{w/z}E_{r/s})$$ which
is a palindrome.

If $rs$ is odd we have, by the induction hypothesis
$$E_{r/s}=E_{x/y}E_{u/v}$$ and by equations~(\ref{eqn:levmnbigger}),~(\ref{eqn:thirdgp}),~(\ref{eqn:fourthgp}) and
(\ref{eqn:fifthgp})
$$E_{m/n}=E_{u/v} E_{r/s}^{(a_{k}-1)}  =E_{u/v}(E_{x/y}E_{u/v})^{(a_{k}-1)} $$ so that
$$E_{p/q}=E_{m/n}E_{r/s}= E_{u/v}(E_{x/y}E_{u/v})^{a_{k} } $$ is a
palindrome.

\vskip .1in

It is clear that the $E_{p/q}$ words are primitive because they are products of associate primitive pairs  and such a product, together with either element of the original primitive pair form a new primitive
pair.  We need, however,  to show that we obtain an $E_{p/q}$ word corresponding to every conjugacy class of primitive words in $F$ (or to the class of its inverse) and that this representative of the  conjugacy class is the unique palindrome in the class when the word  is even.

We work with positive rationals between $0$ and $1$ and note that our comments can easily be adjusted to apply to all rationals.

We observe that in generating the words  $E_{p/q}$ for the positive rationals between $0$ and $1$, there is no cancelation because the $A$'s all appear with negative exponents  and the $B$'s all appear with exponent one.
 Thus there is no difference between concatenation and multiplication followed by free reduction.  It follows that the $E_{p/q}$ have minimal length in their conjugacy class.

In fact, every $E_{p/q}$ with $0 \le p/q < 1$ begins and ends with a power of $A^{-1}$ and the powers of $A^{-1}$
are all separated by a factor of $B$ with exponent $+1$. It is easy to check this for all words of level up through three. The rest follows by induction on the level of $p/q$:  if  $E_{p/q}= E_{r/s}E_{m/n}$ with the levels of $m/n$ and $r/s$ greater than three, the statement holds by the induction hypothesis. Assume that one of the factors is at a level smaller than three. If $E_{p/q}$ is a palindrome, whatever its level, it must start and end with an $A^{-1}$. The only time the factor $E_{m/n}$ or $E_{r/s}$ will not begin and end with an $A^{-1}$ is if it is $E_{1/1}$. In that case we must have $pq$ even, $E_{m/n} = E_{1/1} = BA^{-1}$ and $E_{r/s}$ beginning and ending with $A^{-1}$ terms.

  Since the $E_{p/q}$  are primitive words,  each is conjugate to one of the form of equation~\ref{eq:form1} in Theorem~\ref{thm:genform}.
  Again because there is no cancelation, the  ordered set of exponents for the $A^{-1}$  terms that appear in $E_{p/q}$ are the  primitive exponents, $\{v_1...,v_{p'}\}$ for some integer $p' > 0$ and there is a  $q'$ with $\Sigma_{i=1}^{p'} v_i= q'$. As there is no cancelation when we concatenate $E_{r/s}E_{m/n}$,  the exponents sums of the $A^{-1}$ terms add as do those of the $B$ terms.  Thus we must have $p'=p$ and $q'=q$.

We have thus established  a bijection between the conjugacy classes of primitives (up to inverse) and  the primitive words  $E_{p/q}$. (See also \cite{GKnew}). 

\vskip .01in

 We claim that the palindromes we obtain are unique in their conjugacy class. We assume  $E_{p/q}$ is a palindrome so that $pq$ is even.  Again, we write things as if $p/q$ is in $[0,1]$ and leave it to the reader to make the adjustments for other rationals.   As above, we have only factors of $A^{-1}$ and $B$. Let $V$ be an arbitrary word in $F$ written in the form of equation~\ref{eq:form1}.
    Suppose $W=VE_{p/q}V^{-1}$ is any word in the conjugacy class of $E_{p/q}$.   We consider all of the possible cases:  if $V$ begins or ends with a power of $B$   so $W$ will have powers of $B$ other than $+1$. If $V$ is a single power of $A$, then $W$ will not be a palindrome. Assume $V$ has an initial (or terminal) segment that coincides with the inverse of an initial (or terminal) segment of $E_{p/q}$. Then either $W$ is of a form that does not have minimal length, $V$  is not a palindrome or $VE_{p/q}V^{-1} = E_{p/q}$.  \end{proof}

\subsection{Proof of Theorem~\ref{mainpos}}

\begin{proof}  The proof of the first statement is by induction on the maximum  of the levels of $p/q$ and
$p'/q'$.  Again we proceed assuming $p/q >0$;  reflecting in the
imaginary axis we obtain the proof for negative rationals.

 At level $1$, the theorem
is clearly true: $(A^{-1},B)$, $(A^{-1},BA^{-1})$ and $(BA^{-1},B)$ are all primitive
pairs.

Assume now that the theorem holds for any pair both of whose levels
are less than $N$ and
let $(p/q,p'/q')$ be a pair of neighbors with $Lev(p/q)=N$.
  Let $m/n, r/s$ be the distinguished neighbors of $p/q$ and assume
$m/n<p/q<r/s$.
 Then $m/n,r/s$ are neighbors and both have level
less than $N$ so that by the induction hypothesis
$(E_{m/n},E_{r/s})$ is a pair of primitive  associates.
  It follows that all of  the pairs
$$(E_{m/n},E_{m/n}E_{r/s}),  \,\, (E_{m/n},E_{r/s}E_{m/n}),$$ $$(E_{m/n}E_{r/s}, E_{r/s}) \mbox{
and } (E_{r/s}E_{m/n}, E_{r/s})$$ are pairs of primitive
associates since we can retrieve the original pair
$(E_{m/n},E_{r/s})$ from any of them. Since $E_{p/q}=E_{r/s}E_{m/n}$ or $E_{p/q}=E_{m/n}E_{r/s}$
we have proved the theorem if $p'/q'$ is one of the
distinguished neighbors.

 If $p'/q'$ is not one of the distinguished neighbors, then either
$p'/q'=(tp+m)/(tq+n)$ for some $t>0$ or $p'/q'=(tp+r)/(tq+s)$
for some $t>0$.   Assume for
definiteness $p'/q'=(tp+m)/(tq+n)$; the argument is the same in
the other case.

 Note that the pairs $p/q, (jp+m)/(jq+n)$ are neighbors
for all $j=1, \ldots,  t$.  We have already
shown $E_{m/n},E_{p/q}$ is a pair of primitive associates.
The argument above applied to this pair shows that
$E_{m/n},E_{(p+m)/(q+m)}$ is also a pair of primitive
associates. Applying the argument $t$ times proves the
theorem for the pair $E_{p'/q'},E_{p/q}$.

We have established that all pairs $E_{p'/q'},E_{p/q}$ with $|pq'-q'p|=1$ are primitive pairs.

 Conversely, suppose the pair $(X,Y)$ generates the group with $X$
conjugate to $E_{p/q}$ and $Y$ conjugate to $E_{p'/q'}$.  By Theorem 7.1
of \cite{GKwords}, parts 8 and 9, there is a rational number  $m/n$
associated to $X$ and a rational number $r/s$ associated to $Y$ such that
these rationals are Farey neighbors (that is, $|ms-rn| =1$).  The primitive exponents of  $X$ and
$Y$  are determined by the continued fraction expansions  of $m/n$ and
$r/s$:  $m$ is the exponent sum of the $A$ terms in $X$,  $r$ is the sum of the exponents of the $A$ terms in $Y$ and $m+n$ and $r+s$ are respectively the minimal
word lengths  of the conjugacy classes.  Since $X$ is conjugate to $E_{p/q}$ and $Y$ is conjugate to
$E_{p'/q'}$, as we saw above at the end of the proof of Theorem~\ref{thmmain}, this implies $m=p,n=q,r=p' \mbox{ and } s=q'$ so that
$|pq'-p'q|=1$.
\end{proof}

An immediate corollary is
\begin{cor} The scheme of Theorem~\ref{thmmain} also gives a scheme for enumerating only
primitive palindromes  and a scheme for enumerating only  primitives
that are canonical palindromic products.
\end{cor}

\section{Examples} \label{section:EX3}

Here we compute
some examples of words in the  enumeration scheme:
\vskip .3in
\begin{tabular}{|c|c|c|c|c|c|}

  \hline

\hline

Fraction & Parents & Parity & $E_{p/q}$ & Parental Product & simplified \\
                 \hline
  $1/2$ & $0/1 \oplus 1/1$  & even   & $E_{1/2}$ & $A^{-1} \cdot BA^{-1}$ & $A^{-1}BA^{-1}$   \\  \hline
  $2/1$ & $1/1 \oplus 1/0$ & even & $E_{2/1}$  & $BA^{-1} \cdot B$  & $BA^{-1}B$ \\            \hline
  $1/3$ & $1/2 \oplus 0/1$ & odd & $E_{1/3}$ & $A^{-1}BA^{-1} \cdot A^{-1}$ & $A^{-1}BA^{-2}$ \\                   \hline
    $2/5$ & $1/3 \oplus 1/2$ & even & $E_{2/5}$ & $A^{-1}BA^{-2}\cdot A^{-1}BA^{-1}$ & $A^{-1}BA^{-3}BA^{-1}$ \\       \hline
  $1/4$ & $1/3 \oplus 0/1$ & even & $E_{1/4}$ & $A^{-1} \cdot A^{-1}BA^{-2}$ & $A^{-2}BA^{-2}$ \\    \hline
  $2/7$ & $1/4 \oplus 1/3$ & even  & $E_{2/7}$ & $A^{-2}BA^{-2} \cdot A^{-1}BA^{-2}$ &  $A^{-2}BA^{-3}BA^{-2}$  \\ \hline
\end{tabular}

\vskip .3in

We illustrate how to find the word $E_{31/9}$.  The continued fraction is $31/9=[3,2,4]$;  note that the distinguished Farey neighbors are $[3,2]=7/2$ and $[3,2,3]=24/7$.    We form the following words indicating $E_{31/9}$ and its neighbors in boldface.

$$ E_{1/1}=BA^{-1}, E_{2/1}=BA^{-1}B. E_{3/1}=E_{1/0}E_{2/1} = B\cdot BA^{-1}B,$$ $$ E_{4/1}=E_{3/1}E_{0/1}=BBA^{-1}B \cdot B $$
$${\bf E_{7/2}}=E_{3/1}E_{4/1}={\bf BBA^{-1}B\cdot BBA^{-1}BB}, $$ $$ E_{10/3}=E_{3/1}E_{7/2}=B^2A^{-1}B^3A^{-1}B^3A^{-1}B^2, $$
$$E_{17/5}=E_{7/2}E_{10/3}=B^2A^{-1}B^3A^{-1}B^2 \cdot   B^2A^{-1}B^3A^{-1}B^3A^{-1}B^2, $$
 $${\bf E_{24/7}}=E_{17/5}\cdot E_{7/2} ={\bf B^2A^{-1}B^3A^{-1}B^2   B^2A^{-1}B^3A^{-1}B^3A^{-1}B^2 \cdot B^2A^{-1}B^3A^{-1}B^2}$$
$${\bf E_{31/9}}= E_{7/2} E_{24/7}= $$
$${\bf B^2A^{-1}B^3A^{-1}B^2 \cdot B^2A^{-1}B^3A^{-1}B^2   B^2A^{-1}B^3A^{-1}B^3A^{-1}B^2  B^2A^{-1}B^3A^{-1}B^2}$$

\section{Farey Diagram Visualization}

 We can visualize the relation between primitive pairs and neighboring rationals using the Farey diagram.
 Suppose
  the primitive
 pair $(A,B)$ corresponds to $(r/s,p/q)$.   Note that we have done this so that the Farey level of
 $p/q$ is greater than that of $r/s$ and $r/s$ is   the parent of $p/q$ with lowest Farey level.
 (The other parent is $(r-p)/(q-s)$ and corresponds to $A^{-1}B$.) Draw the   curve $\gamma$ from a point on
 the imaginary axis to $p/q$.  If $p/q$ is positive, orient it toward $p/q$; if $p/q$ is negative,
 orient it towards the imaginary axis.

 The left-right sequence and the continued fraction for $p/q$
 are  the same.  Traversing the curve in the other direction reverses
 left and right.
  The symmetry about the imaginary axis is
 reflected in our definition of negative continued fractions.

 Given two primitive pairs $(A,B)$ and $(A'B')$ such that $B$
corresponds to $p/q$ and $B'$ corresponds to $p'/q'$ draw the curves
$\gamma$ and $\gamma'$.
 We can find the sequence to go from $(A,B)$ to $(A',B')$ by traversing $\gamma$ from $p/q$ to the
 imaginary axis and then traversing $\gamma'$ to $p'/q'$.   We can also draw an oriented curve
 from $p/q$ to $p'/q'$ and read off the left-right sequence along this curve to get  the
 continued fraction
 that gives $(A',B')$ as words in $(A,B)$ directly. We have

 \begin{cor} Given any two sets of primitive pairs, $(A,B)$
 and $(A',B')$ there is a continued fraction containing either only
 positive or only negative integers that connects one pair to the
 other.
 \end{cor}

\vskip .1in  We thank Vidur Malik whose thesis \cite{Malik, Malik2} suggested that we look
at palindromes. We also thank the referee who careful reading and helpful suggestions has done much to improve the paper.

\bibliographystyle{amsplain}

\begin{thebibliography}{99}

\bibitem{BShpilT} Bardakov, Sprilrain and Tolytykh, {\em On the palindromic and primitive widths in a free group}, J. Algebra 285 (2005) 574-585.

\bibitem{BuserS} Buser, P.; Semmler, K.-D.
{\em The geometry and spectrum of the one-holed torus.}  Comment.
Math. Helv.  63  (1988),  no. 2, 259-274.


\bibitem{CMZ} Cohen, P;  Metzler, W; Zimmermann, B. {\em  What does a basis of $F(a, b)$
look like?}, Math. Ann. 257 (4) (1981) 435-445.


\bibitem{G} Gilman, Jane, {\em Informative Words and Discreteness}, Contemp. Math  {\bf 421},
   (2007) 147-155.
Cont Math.

\bibitem{G1} Gilman, Jane,  {\em  Algorithms, Complexity
and Discreteness Criteria in $PSL(2,{\mathbb{C}})$},
 Journal
D'Analyse Mathematique, Vol 73, (1997), 91-114.


\bibitem{GKwords} Gilman, Jane and Keen, Linda, {\em Word sequences and intersection numbers}.
Complex manifolds and hyperbolic geometry (Guanajuato, 2001),
231-249, Contemp. Math., 311, Amer. Math. Soc., Providence, RI, 2002

\bibitem{GKHarvey} Gilman, Jane and Keen, Linda, {\em Cutting Sequences and Palindromes,Geometry of Riemann Surfaces}
Series: London Mathematical Society Lecture Note Series (No. 368), 194-216

\bibitem{GKgeom} Gilman, Jane and Keen, Linda, {\em Discreteness  Criteria and the Hyperbolic Geometry of Palindromes},
Conformal Geometry and Dynamics, 13  (2009), 76-90  arXiv:0808.3488,

\bibitem{GKnew} Lifting from $PSL(2, \RR)$ to the free group $F=<a,b>$, in preparation.



\bibitem{GKMORE} Hyperbolic Geometry and Core Geodesics, in preparation.

\bibitem{GM} Gilman, J.  and Maskit, B. {\em An algorithm for 2-generator Fuchsian groups}
Michigan Math. J. {\bf 38} (1991) 13-32.
\bibitem{HardyWright} Hardy, G.H. and  Wright, E. M., {\em An introduction to the theory of
numbers}
 Oxford, Clarendon Press 1938.


\bibitem{YCJ} Jiang, Yicheng, {\em Polynomial complexity of the Gilman-Maskit
discreteness algorithm}  Ann. Acad. Sci. Fenn. Math.  26  (2001),
no. 2, 375-390.

\bibitem{KR} Kassel, Christian; Reutenauer, Christophe {\em Sturmian morphisms, the
braid group $B\sb 4$, Christoffel words and bases of $F\sb 2$}. Ann.
Mat. Pura Appl. (4)  186  (2007),  no. 2, 317-339.

\bibitem{KS} Keen, Linda and Series, Caroline, {\em  Pleating
Coordinates for the Maskit Embedding of Teichm\"{u}ller space for a
punctured torus},  Topology, Vol. {\bf 32} \#4, (1993), 719-749.


\bibitem{Malik} Malik, Vidur, {\em Curves Generated on Surfaces by
the Gilman-Maskit Algorithm}, Ph D thesis, Rutgers University,
Newark, NJ (2007).
\bibitem{Malik2} Malik, Vidur {\em Primitive words and self-intersections of curves on surfaces generated by the Gilman-Maskit discreteness algorithm}, In the tradition of Ahlfors-Bers. V, Contemp. Math., 510, Amer. Math. Soc., Providence, RI, (2010), 209-237.


\bibitem{MKS} Magnus, Wilhelm; Karrass, Abraham; Solitar, Donald {\sl Combinatorial group theory: Presentations of groups in terms of generators and relations} Interscience Publishers [John Wiley \& Sons, Inc.], New York-London-Sydney (1966).


\bibitem{OZ} Osborne, R. P.; Zieschang, H. {\sl Primitives in the free group on two
generators} Invent. Math. 63 (1981), no. 1, 17-24.

\bibitem{Piggott}Piggott, Adam,  {\sl Palindromic primitives and palindromic bases in the
free group of rank two}  J. Algebra 304 (2006), no. 1, 359-366.


\bibitem{CS3} Series, Caroline, {\em Non-euclidean geometry, continued fractions and ergodic theorey,}  Math. Intelligencer
\#4,(1982), 24-31.




\bibitem{CS2} Series, Caroline
{\em The modular surface and continued fractions}. J. London Math. Soc. 2, {\bf 31} (1985), 69-8.


\bibitem{Series} Series, Caroline
{\em The Geometry of Markoff Numbers}, Math. Intelligencer {\bf 7}
\#3,(1985), 20-29.



\bibitem{Vin} Vinogradov, I.M. {\em An Introduction to the Theory of Numbers},
(Enlgish translation by H. Popova) Pergamon Press, London-NY, (1955).
MR \# 0070644.
\bibitem{Wr} Wright, David J., {\sl Searching for the cusp}, Spaces
of Kleinian Groups, LMS Lecture Notes {\bf 329}, Cambridge U. Press
(2004), 1-36.
\end{thebibliography}

\end{document}